\documentclass[final,nomarks]{dmtcs-episciences}

\usepackage[utf8]{inputenc}
\usepackage{subfigure}

\usepackage[round]{natbib}

\usepackage{url}
\usepackage{amssymb}
\usepackage{amscd}
\usepackage{amsmath}
\usepackage{amsthm}
\usepackage{amsfonts}
\usepackage[normalem]{ulem}
\newcommand{\stkout}[1]{\ifmmode\text{\sout{\ensuremath{#1}}}\else\sout{#1}\fi}
\newtheorem{theorem}{Theorem}
\newtheorem{lemma}{Lemma}

\theoremstyle{definition}
\theoremstyle{definition}
\theoremstyle{definition}\newtheorem{definition}{Definition}

\author{James Currie\thanks{The work of James Currie is supported by the Natural
  Sciences and Engineering Research Council of Canada (NSERC), [funding reference number DDG-2024-00005].}
  \and Narad Rampersad\thanks{The work of Narad Rampersad is supported by the Natural
  Sciences and Engineering Research Council of Canada (NSERC), [funding reference number RGPIN-2019-04111].}}
\title{Low complexity binary words avoiding $(5/2)^+$-powers}
\affiliation{
  University of Winnipeg, Winnipeg, Canada}
\begin{document}
\publicationdata{vol. 27:3}{2025}{17}{10.46298/dmtcs.15939}{2025-06-25; 2025-06-25; 2025-10-13}{2025-10-15}
\maketitle
\begin{abstract}
Rote words are infinite words that contain $2n$ factors of length $n$ for every $n \geq 1$.
Shallit and Shur, as well as Ollinger and Shallit, showed that there are Rote words
that avoid $(5/2)^+$-powers and that this is best possible.  In this note we give a structure
theorem for the Rote words that avoid $(5/2)^+$-powers, confirming a conjecture of Ollinger
and Shallit.
\end{abstract}

\section{Introduction}
Two central concepts in combinatorics on words are \emph{power avoidance} and \emph{factor complexity}.
Recently, \cite{shur19} initiated the systematic investigation of the interplay between
these two concepts.  They examined two dual problems: 1) Given a particular power avoidance constraint,
determine the range of possible factor complexities among all infinite words avoiding the specified power;
and, 2) Given a class of words with specified factor complexities, determine the powers that are avoided
by at least one word in this class.

A well-known classical result provides a solution for the latter problem for the class of \emph{Sturmian words},
i.e., the class of infinite words that contain $n+1$ factors of length $n$ for every $n \geq 1$:
the \emph{Fibonacci word} avoids $(5+\sqrt{5})/2$-powers, and this is best possible among all Sturmian words (\cite{carpi2000}).
The Sturmian words are the aperiodic infinite words with the least possible factor complexity function;
\cite{shur19} and \cite{shallit2024} studied another class of
infinite words with low complexity, namely, the \emph{Rote words}, which are the infinite words that contain
$2n$ factors of length $n$ for every $n \geq 1$ (\cite{rote1994}).  Each paper gives an example of an infinite Rote word
that avoids $(5/2)^+$-powers and shows that this is best possible among all Rote words.
Ollinger and Shallit end their paper by observing that the Rote words that avoid $(5/2)^+$-powers
appear to have a certain rigid structure reminiscent of the famous structure theorem of \cite{restivo1985}
for the class of overlap-free words. In this note 
we obtain the precise structure theorem.

To state our first structure theorem, we introduce {\em proper words} and {\em antiproper words}, which are ternary words.
\begin{definition}
For $u\in\Sigma_3^*$, denote the Parikh vector of $u$ by $\pi(u)$so that $\pi(u)=[|u|_0,|u|_1,|u|_2].$ For $x,y\in\Sigma_3^*$, we say that $\pi(x)>\pi(y)$ 
if:
\begin{enumerate}
\item We have $|x|_i\ge|y|_i$ for all $i\in\Sigma_3$. 
\item For at least one 
$i\in\Sigma_3$ we have $|x|_i>|y|_i.$
\end{enumerate}
 Call a word $u\in\Sigma_3^*$ {\em proper} if:
\begin{enumerate}
\item Word $u$ has no factor $xyxyx$ where $\pi(x)>\pi(y)$.
\item None of the words $00$, $11$, $22$, $20$, $10101$,  $2121$, or $10210210$ is a factor of $u$.
\end{enumerate}
Call a word ${\pmb u}\in\Sigma_3^\omega$ proper if all of its finite factors are proper.

Call a word $u\in\Sigma_3^*$ {\em antiproper} if its reverse $u^R$ is proper. Call a word ${\pmb u}\in\Sigma_3^\omega$ antiproper if all of its finite factors are antiproper.
\end{definition}

Proper words obey a structure theorem similar to that of Restivo and Salemi, as do antiproper words. We introduce a morphism $f:\Sigma_3^*\rightarrow \Sigma_3^*$ and its reverse $h:\Sigma_3^*\rightarrow \Sigma_3^*$, given by:

\begin{align*}
f(0)&=0121&h(0)&=1210\\
f(1)&=021&h(1)&=120\\
f(2)&=01&h(2)&=10.
\end{align*}

\begin{theorem}\label{forcing inner f}(First Structure Theorem)
\begin{enumerate}
\item Let ${\pmb u}\in\Sigma_3^\omega$ be proper.
Then a final segment of ${\pmb u}$ has the form $f({\pmb v})$ for some proper  ${\pmb v}\in\Sigma_3^\omega$.
\item Let ${\pmb u}\in\Sigma_3^\omega$ be antiproper.
Then a final segment of ${\pmb u}$ has the form $h({\pmb v})$ for some antiproper  ${\pmb v}\in\Sigma_3^\omega$.
\end{enumerate}
\end{theorem}

For our second structure theorem we consider the length-4 factors of a Rote word.
\begin{definition} Let $F$ be the set
$$F=\{0110, 1001, 0011, 1100, 0010, 0100, 1101, 1010\}.$$
Let $g:\Sigma_3^*\rightarrow \Sigma_2^*$ be the morphism given by
\begin{align*}
g(0)&=011\\
g(1)&=0\\
g(2)&=01.
\end{align*}
We denote the complement of a binary word $w$ by $\overline{w}$; thus $\overline{1101}=0010$. We extend this notation to binary languages in the usual way. We denote the reversal of a word $w$ by $w^R$; thus $1101^R=1011$. We extend this notation to languages in the usual way. 

\end{definition}
We can now characterize the structure of Rote words that avoid $(5/2)^+$-powers.

\begin{theorem}\label{forcing outer g}(Second Structure Theorem)
Let ${\pmb w}$ be a Rote word that avoids $(5/2)^+$-powers. There are four cases:
\begin{enumerate}
\item The set of length-4 factors of ${\pmb w}$ is $F$.  For every positive integer $n$, a final segment of ${\pmb w}$ has the form $g(f^n({\pmb u}))$ for some proper ${\pmb u}\in\Sigma_3^\omega$. 
\item The set of length-4 factors of ${\pmb w}$ is $\bar{F}$.  For every positive integer $n$, a final segment of ${\pmb w}$ has the form $\overline{g(f^n({\pmb u}))}$ for some proper ${\pmb u}\in\Sigma_3^\omega$. 
\item The set of length-4 factors of ${\pmb w}$ is $F^R$.  For every positive integer $n$, a final segment of ${\pmb w}$ has the form $g(h^n({\pmb u}))$ for some proper ${\pmb u}\in\Sigma_3^\omega$. 
\item The set of length-4 factors of ${\pmb w}$ is $\overline{F^R}$.  For every positive integer $n$, a final segment of ${\pmb w}$ has the form $\overline{g(h^n({\pmb u}))}$ for some proper ${\pmb u}\in\Sigma_3^\omega$. 
\end{enumerate}
\end{theorem}

\section{Preliminaries}

For a positive integer $n$, let $\Sigma_n=\{0,1,\ldots, n-1\}$ and let $\Sigma_n^*$ denote the set of
all finite words over $\Sigma_n$.  By a binary word we mean a (finite or infinite) word over $\Sigma_2$.
Let $w$ be a word and write $w=xyz$.  Then the word $x$ is a \emph{prefix} of $w$, the word $y$ is a \emph{factor}
of $w$, and the word $z$ is a \emph{suffix} of $w$.  A map $f:\Sigma_m^*\to\Sigma_n^*$ is a \emph{morphism}
if $f(xy)=f(x)f(y)$ for all $x$ and $y$.

Let $w$ have length $\ell$ and smallest period $p$.  The \emph{exponent of $w$} is the quantity $k=\ell/p$
and $w$ is called a \emph{$k$-power}.  A \emph{$k^+$-power} is a word with exponent $>k$.  A $2$-power is
a \emph{square} and a $2^+$-power is an \emph{overlap}.  A word $w$ \emph{avoids $k$-powers}
(resp.\ \emph{$k^+$-powers}) if none of its factors are $k'$-powers for any $k'>=k$ (resp. $k'>k$);
we also say that $w$ is \emph{$k$-power-free} (resp.\ \emph{$k^+$-power-free}).

Let ${\bf x}$ be an infinite word.  The \emph{factor complexity of ${\bf x}$} is the function of $n$
that associates each length $n$ with the number of factors of ${\bf x}$ of length $n$.  A \emph{Sturmian word}
is any infinite word with factor complexity $n+1$; a \emph{Rote word} is any infinite word with factor complexity $2n$.

Recently the morphisms $f$ and $g$ have proved useful in several constructions (\cite{currie2023,ochem2024,shallit2024}). 
\cite{ochem2024} showed that $g(f^\omega(0))$ avoids $(5/2)^+$-powers, while \cite{shallit2024}
showed that the factor complexity of this word is $2n$ (i.e., that it is a Rote word).
The latter authors conjectured that there is a structure theorem involving $f$ and $g$ for the class of
$(5/2)^+$-power-free Rote words.

A prototypical structure theorem of this type was obtained by \cite{restivo1985} for the
class of binary overlap-free words.  (See also \cite{fife80,shur96,shallit11}.) In this case, the structure is specified using the Thue--Morse morphism
\begin{align*}
\mu(0)&=01\\
\mu(1)&=10.
\end{align*}
Restivo and Salemi showed the following:
\begin{theorem}Let ${\pmb w}\in\Sigma_2^\omega$ be overlap-free. Then a final segment of ${\pmb w}$ has the form $\mu({\pmb u})$
where ${\pmb u}$ is overlap-free.
\end{theorem}
\cite{karhumaki2004} later showed that the same structure theorem holds for the class of binary
$(7/3)^+$-power-free words.  In this note we establish a similar structure theorem for the class of $(5/2)^+$-power-free Rote words
in terms of the morphisms $f$ and $g$ given above.  The reader may also compare the present structure theorem with the main result
of \cite{currie2020}, which establishes a structure theorem for the class of infinite $14/5$-power-free binary rich words
(\emph{rich} means that every factor of length $n$ contains $n$ distinct non-empty palindromes),
and which also involves a sub-family of Rote words, namely the \emph{complementary symmetric Rote words}.

\section{Obtaining the structure theorem}

\begin{lemma}\label{0110}Let ${\pmb w}$ be an infinite binary word
which avoids $\frac{5}{2}^+$ powers. Then both of the words $0110$ and $1001$ are factors of ${\pmb w}$.
\end{lemma}
\begin{proof} 
A backtrack search shows that the longest $\frac{5}{2}^+$ power free binary word not containing $0110$ has length 14. Thus $0110$ (and symmetrically, $1001$) is a factor of ${\pmb w}$.
\end{proof}
\begin{lemma}\label{0010}Let ${\pmb w}$ be an infinite binary word
which avoids $\frac{5}{2}^+$ powers. At least 3 of the words in
$$C=\{0010,0100,1011,1101\}$$
are factors of ${\pmb w}$.
\end{lemma}
\begin{proof}
Six backtrack searches (one for each pair) show that the longest $\frac{5}{2}^+$ power free binary word omitting a pair of these factors has length 44.
\end{proof}
\begin{lemma}\label{0011}Let ${\pmb w}$ be an infinite binary word
 with factor complexity at most $2n$, 
which avoids $\frac{5}{2}^+$ powers. Both of $0011$ and $1100$
are factors of ${\pmb w}$.
\end{lemma}
\begin{proof}
Consider the set of seven binary words
$$A=\{0010,0100,0101,1010,1011,1101,1100\}.$$
For each word $a\in A$, a backtrack search shows that the longest $\frac{5}{2}^+$ power free binary word containing neither of $a$ and $0011$
as a factor has length no more than 31.

It follows that if $0011$ is not a factor of ${\pmb w}$, then ${\pmb w}$ contains the seven words of $A$ as length 4 factors. By Lemma~\ref{0110}, it also contains 
$0110$ and $1001$ as factors. However, now ${\pmb w}$ contains 9 factors of length 4, contradicting the fact that its factor complexity is at most $2n$.

We conclude that $0011$ (and symmetrically, $1100$) is a factor of  ${\pmb w}$.
\end{proof}
\begin{lemma}\label{0101}Let ${\pmb w}$ be an infinite binary word
 with factor complexity at most $2n$, 
which avoids $\frac{5}{2}^+$ powers. At least one of $0101$ and $1010$
is a factor of ${\pmb w}$.
\end{lemma}
\begin{proof}
Consider the set $D$ containing 17 binary words of length 9 given by
$$D=\{00100110,
 01001100,
 10011001,
 00110010,
 01100100,
 11001001,$$$$
 10010011,
 00110011,
 01100110,
 11001101,
 10011011,
 00110110,$$$$
 01101100,
 11011001,
 10110010,
 10110011,
 11001100\}.$$
For each word $d\in D$, a backtrack search shows that the longest $\frac{5}{2}^+$ power free binary word containing none of $d$, $0101$, and $1010$
as a factor has length no more than 88. It follows that if neither of $0101$ and $1010$
is a factor of ${\pmb w}$, then every word of $D$ is a factor. However, this would imply that ${\pmb w}$ contained 17 factors of length 8, contradicting the fact that its factor complexity is at most $2n$. We conclude that
at least one of $0101$ and $1010$
is a factor of ${\pmb w}$.
\end{proof}
\begin{theorem} Let ${\pmb w}$ be an infinite binary word
 with factor complexity at most $2n$, 
which avoids $\frac{5}{2}^+$ powers. Up to binary complement and/or reversal, the set of length 4 factors of ${\pmb w}$ is
$$\{0110, 1001, 0011, 1100, 0010, 0100, 1101, 1010\}.$$ 
\end{theorem}
\begin{proof}
By Lemma~\ref{0110}, the set of length 4 factors includes $0110$ and $1001$.
By Lemma~\ref{0011}, the set of length 4 factors includes $0011$ and $1100$.
Combining Lemmas \ref{0010} and \ref{0101} with the fact that ${\pmb w}$ has at most 8 length 4 factors, 
the set of length 4 factors contains exactly 3 words from $C=\{0100,0010,1011,1101\}$ and exactly one word from $\{0101,1010\}$.
Since each word of $C$ maps to each of the others under complement and/or reversal, assume without loss of generality that the 3 words from $C$ are
$0010$, $0100$, and $1101$. Thus ${\pmb w}$ does not contain the factor $1011$.

A backtrack search shows that the longest $\frac{5}{2}^+$ power free binary word not containing $1011$ or $1010$
as a factor has length 20. We conclude that ${\pmb w}$ contains the factor $1010$, so that set of length 4 factors of ${\pmb w}$ is
$$\{0110, 1001, 0011, 1100, 0010, 0100, 1101, 1010\}.$$ 
\end{proof}
\begin{lemma}\label{xyxyx bin} Suppose that $u\in\Sigma_3^*$,  $w\in\Sigma_2^*$, and $g:\Sigma_3^*\rightarrow \Sigma_2^*$ is a non-erasing morphism.
If $g(u)=w$, and $w$ avoids $\frac{5}{2}^+$ powers, then $u$ has no factor $xyxyx$ where $\pi(x)>\pi(y)$.
\end{lemma}
\begin{proof} Suppose $u$ has a factor $xyxyx$ where $\pi(x)>\pi(y)$. Then $|g(x)|>|g(y)|$, and $w$ contains the 
$\frac{5}{2}^+$ power $g(x)g(y)g(x)g(y)g(x)$.
\end{proof}
\begin{lemma}\label{xyxyx tern} Suppose that $u,v\in\Sigma_3^*$,  and $f:\Sigma_3^*\rightarrow \Sigma_3^*$ is a non-erasing morphism.
If $f(v)=u$, and $u$ has no factor $xyxyx$ where $\pi(x)>\pi(y)$, then $v$ has no factor $XYXYX$ where $\pi(X)>\pi(Y)$.
\end{lemma}
\begin{proof} Suppose $v$ has a factor $XYXYX$ where $\pi(X)>\pi(Y)$. Let $x=f(X)$ and $y=f(Y)$. Then $\pi(x)>\pi(y)$, and $u$ contains the factor $xyxyx$.
\end{proof}
\begin{lemma}\label{lemma forcing outer g}
Let ${\pmb w}$ be an infinite binary word
which avoids $\frac{5}{2}^+$ powers. Suppose that the set of length 4 factors of ${\pmb w}$ is
$$F=\{0110, 1001, 0011, 1100, 0010, 0100, 1101, 1010\}.$$
Then a final segment of ${\pmb w}$ has the form $g({\pmb u})$ for some proper ${\pmb u}\in\Sigma_3^\omega$.
\end{lemma}
\begin{proof} Since $111$ is not a factor of ${\pmb w}$, any final segment of ${\pmb w}$ beginning with $0$ has the form
$g({\pmb u})$ for some ${\pmb u}\in\Sigma_3^\omega$. We will show, replacing ${\pmb u}$ by one of its final segments if necessary, that 
${\pmb u}$ is proper. 

The fact that ${\pmb u}$ has no factor $xyxyx$ where $\pi(x)>\pi(y)$ follows from Lemma~\ref{xyxyx bin}.

We conclude by showing that none of the words $00$, $11$, $22$, $20$, $10101$, $2121$, or $10210210$ is a factor of ${\pmb u}$:\vspace{.1in}

\noindent{\bf Word $00$:} If $00$ is a factor of ${\pmb u}$, then ${\pmb w}$ contains the factor $g(00)=011011$. However, then ${\pmb w}$ contains $1011$, which is not in $F$.\vspace{.1in}

\noindent{\bf Word $11$:} If $11$ is a factor of ${\pmb u}$, then $11a$ is a factor of ${\pmb u}$ for some $a\in\Sigma_3$. Then ${\pmb w}$ contains the factor $g(11a)$, which starts with $g(1)g(1)0=000$, a $\frac{5}{2}^+$ power.\vspace{.1in}

\noindent{\bf Word $22$:} If $22$ is a factor of ${\pmb u}$, then ${\pmb w}$ contains the factor $g(22)=0101\notin F$.\vspace{.1in}

\noindent{\bf Word $20$:} If $20$ is a factor of ${\pmb u}$, then ${\pmb w}$ contains the factor $g(20)=01011$, which starts with $0101\notin F$.\vspace{.1in}

\noindent{\bf Word $10101$:} If $10101$ is not a factor of ${\pmb u}$ more than once, replace ${\pmb u}$ by one of its final segments not containing 
$10101$. 

On the other hand, if $10101$ is a factor of ${\pmb u}$ more than once, then $a10101b$ is a factor of ${\pmb u}$ for some $a,b\in\Sigma_3$. Since $11$ is not a factor of ${\pmb u}$, we can in fact specify that $a\in\Sigma_3-\{1\}.$ This implies that $1$ is the last letter of $g(a)$. Also, $0$ is the first letter of $g(b)$. Then ${\pmb w}$ contains the factor $g(a10101b)$, which 
contains $1g(10101)0=10011001100$, a $\frac{5}{2}^+$ power.\vspace{.1in}

\noindent{\bf Word $2121$:} If $2121$ is a factor of ${\pmb u}$, then $2121a$ is a factor of ${\pmb u}$ for some $a\in\Sigma_3-\{1\}$. Thus $01$ is a prefix of $g(a)$, and ${\pmb w}$ contains the factor $g(2121)01=01001001$, a $\frac{5}{2}^+$ power.\vspace{.1in}

\noindent{\bf Word $10210210$:} If $10210210$ is a factor of ${\pmb u}$, then ${\pmb w}$ contains the factor $$g(10210210)=0011010011010011,$$ a $\frac{5}{2}^+$ power.
\end{proof}

We can now prove Theorem~\ref{forcing inner f}:

\begin{proof}[of Theorem~\ref{forcing inner f}] Assume without loss of generality, replacing ${\pmb u}$ by a final segment if necessary, that ${\pmb u}$ starts with the letter $0$. 
Write ${\pmb u}$ as a concatenation of $0$-blocks, i.e.,  words which start with $0$, and contain the letter $0$ exactly once. Since ${\pmb u}$ is proper, it does not contain a factor $00$ or $20$. It follows that its $0$-blocks end with $1$. Since every occurrence of $2$ in a proper word can only be followed by a $1$, ${\pmb u}$ cannot contain the factor $212$; otherwise it would contain the forbidden factor $2121$. We conclude that the $0$-blocks of ${\pmb u}$ are among
$01$, $021$ and $0121$. Notice that the last two of these words are always followed by a $0$ in ${\pmb u}$. We thus conclude that
a final segment of ${\pmb u}$ has the form $f({\pmb v})$ for some ${\pmb v}\in\Sigma_3^\omega$.
We will show, replacing ${\pmb v}$ by one of its final segments if necessary, that 
${\pmb v}$ is proper. 

The fact that ${\pmb v}$ has no factor $xyxyx$ where $\pi(x)>\pi(y)$ follows from Lemma~\ref{xyxyx tern}.

We conclude by showing that a final segment of ${\pmb v}$ contains none of the words $22$, $20$, $00$, $11$, $10101$, $2121$, or $10210210$ as a factor:\vspace{.1in}

\noindent{\bf Word $22$:} If $22$ is not a factor of ${\pmb v}$ more than once, then replace ${\pmb v}$ by one of its final segments not containing $22$. Otherwise, $22$ is a factor of ${\pmb v}$ more than once, so that ${\pmb u}$ contains a factor $1f(22)=10101.$ This is impossible, since ${\pmb u}$ is proper.\vspace{.1in}

\noindent{\bf Word $20$:} If $20$ is not a factor of ${\pmb v}$ more than once, then replace ${\pmb v}$ by one of its final segments not containing $20$. Otherwise, $20$ is a factor of ${\pmb v}$ more than once, so that ${\pmb u}$ contains a factor $1f(20)=1010121$, which starts with $10101$. This is impossible, since ${\pmb u}$ is proper.\vspace{.1in}

\noindent{\bf Word $00$:} Suppose that $00$ is a factor of ${\pmb v}$. If $000$ is a factor of ${\pmb v}$, then $f(000)=012101210121$ is a factor of ${\pmb u}$. However $012101210121=xyxyx$ where $x=0121$, $y=\epsilon$, and cannot be a factor of ${\pmb u}$. It follows that $000$ is not a factor of ${\pmb v}$. 

If $00$ is not a factor of ${\pmb v}$ more than once, then replace ${\pmb v}$ by one of its final segments not containing $00$. Otherwise, $00$ is a factor of ${\pmb v}$ more than once, so that $100a$ is a factor of ${\pmb v}$ for some $a\in\Sigma_3$. Since $f(a)$ starts with $0$, this implies that ${\pmb u}$ has a factor $f(100a)$, starting with $021012101210$. This contains the word $xyxyx$ where $x=210$, $y=1$, which is impossible, since ${\pmb u}$ is proper.\vspace{.1in}

\noindent{\bf Word $11$:} If $11$ is not a factor of ${\pmb v}$ more than once, then replace ${\pmb v}$ by one of its final segments not containing $11$. Otherwise, $11$ is a factor of ${\pmb v}$ more than once, so that ${\pmb u}$ contains a factor $1f(11)0=10210210$, which is impossible, since ${\pmb u}$ is proper.\vspace{.1in}

\noindent{\bf Word $10101$:} If $10101$ is not a factor of ${\pmb v}$ more than once, then replace ${\pmb v}$ by one of its final segments not containing $10101$. Otherwise, $10101$ is a factor of ${\pmb v}$ more than once, so that ${\pmb u}$ contains a factor $1f(10101)0=10210 12 10210 12 10210=xyxyx$, where $x=10210$ and $y=12$. This is impossible, since ${\pmb u}$ is proper.\vspace{.1in}

\noindent{\bf Word $2121$:} If $2121$ is not a factor of ${\pmb v}$ more than once, then replace ${\pmb v}$ by one of its final segments not containing $2121$. Otherwise, $2121$ is a factor of ${\pmb v}$ more than once. Since $22$ is not a factor of $v$, $a2121b$ is a factor of ${\pmb v}$  for some $a,b\in\Sigma_3$ where $a\ne 2$. Then $f(a)$ ends in $21$ and $f(b)$ begins with $0$. Thus ${\pmb u}$ contains the factor $21f(2121)0=2101021010210=xyxyx$ where $x=210$, and $y=10$, which is impossible, since ${\pmb u}$ is proper.\vspace{.1in}

\noindent{\bf Word $10210210$:} If $10210210$ is a factor of ${\pmb v}$, then ${\pmb u}$ contains the factor $$f(10210210)=0210121 01 0210121 01 0210121=xyxyx,$$ where $x=0210121$, and $y=01$, which is impossible, since ${\pmb u}$ is proper.

The proof for antiproper words is the same, {\em mutatis mutandi}.
\end{proof}

We can now prove Theorem~\ref{forcing outer g}:

\begin{proof}[of Theorem~\ref{forcing outer g}] The first case follows from Theorem~\ref{forcing inner f} and Lemma~\ref{lemma forcing outer g} by induction. The other cases follow, {\em mutatis mutandi}.
\end{proof}

As an example of the third case of the theorem, \cite{shur19} consider a word $\tau({\bf G})$, where $\tau:\Sigma_3^*\rightarrow \Sigma_2^*$ is the morphism given by
\begin{align*}
\tau(0)&=0\\
\tau(1)&=01\\
\tau(2)&=011
\end{align*}
and ${\bf G}$ is the fixed point of $\theta$, where $\theta:\Sigma_3^*\rightarrow \Sigma_3^*$ is the morphism given by
\begin{align*}
\theta(0)&=01\\
\theta(1)&=2\\
\theta(2)&=02.
\end{align*}
Letting $\sigma$ be the permutation $0\to 1\to 2\to 0$, one checks that 
$\tau=g\sigma$ and $\theta^2=\sigma^{-1}h\sigma$. It follows that
\begin{align*}
\tau({\bf G})&=\tau(\theta^{2n}(\theta^\omega(0)))\\
&=g\sigma((\sigma^{-1}h\sigma)^{n-1}(\sigma^{-1}h\sigma(\theta^\omega(0)))\\
&=gh^{n-1}({\pmb u})),
\end{align*} 
where ${\pmb u}=h(\sigma(\theta^\omega(0)))$, which is antiproper.

Many open problems remain concerning the relationship between low factor complexity and avoidable powers.
Moving to a ternary alphabet, \cite{shur19} showed that the word $${\bf G} = 0120201020120102012\cdots$$
has critical exponent $2.4808627\cdots$ and factor complexity $2n+1$ for all $n \geq 1$.  They conjectured that
this exponent is minimal among all infinite ternary words with complexity $2n+1$. This conjecture was recently confirmed by \cite{currie2025}.  It would thus be natural to explore the possibility
of a structure theorem for this class of words.
\section{Appendix: Python code and output}
The backtrack searches mentioned in the paper run quickly in Python. Here is our code and its output:
{\footnotesize
\begin{verbatim}
def fhpf(w): #Word w is 5/2+ power suffix free
    p=1
    while (5*p<2*len(w)):
        if (w[(-(p+1)//2)-p:]==w[(-(p+1)//2)-2*p:-p]):
            return(False)
        p=p+1
    return(True)

def good(w): # Word w has no suffix which is a 5/2+ power, or is in the
             # set Factors.
    for f in Factors:
        k=len(f)
        if ((len(w)>= k)and(w[-k:]==f)):
            return(False)
    return(fhpf(w))

def search(target): # This returns the lexicographically least word not
                    # containing a 5/2+ power or a word in the set Factors
    w=''
    Max=0
    while (len(w)<=target):
        if (good(w)):
            Max=max(Max,len(w))
            if (len(w)==target):
                return(w)
            w+='0'
        else:
            while((len(w)>0) and (w[-1]=='1')):
                w=w[:-1]
            if(len(w)==0):
                print('Longest 5/2+-power-free word with no factor in ',Factors,'
                    has length ',Max)
                return()
            c=chr(ord(w[-1])+1)          
            w=w[:-1]
            w+=c
    return()    
       

# Lemma 1

print('=============')
print('Computations for Lemma 1')
print(' ')

Factors=['0110']
search(200)

# Lemma 2

print(' ')
print('=============')
print('Computations for Lemma 2')
print(' ')

C=['0010','0100','1011','1101']
for i in range(4):
    for j in range(i,4):
        Factors=[C[i]]
        Factors.append(C[j])
        search(200)


# Lemma 3

print(' ')
print('=============')
print('Computations for Lemma 3')
print(' ')

A=['0010','0100','0101','1010','1011','1101','1100']
for j in A:
    Factors=['0011']
    Factors.append(j)
    search(200)


# Lemma 4

print(' ')
print('=============')
print('Computations for Lemma 4')
print(' ')

D=['00100110','01001100','10011001','00110010',
 '01100100','11001001','10010011','00110011',
 '01100110','11001101','10011011','00110110',
 '01101100','11011001','10110010','10110011','11001100']
for j in D:
    Factors=['0101','1010']
    Factors.append(j)
    search(200)

# Theorem 4

print(' ')
print('=============')
print('Computations for Theorem 4')
print(' ')

Factors=['1011','1010']
search(200)

==============================================================================

=============
Computations for Lemma 1
 
Longest 5/2+-power-free word with no factor in  ['0110']  has length  14
 
=============
Computations for Lemma 2
 
Longest 5/2+-power-free word with no factor in  ['0010', '0100']  has length  44
Longest 5/2+-power-free word with no factor in  ['0010', '1011']  has length  28
Longest 5/2+-power-free word with no factor in  ['0010', '1101']  has length  13
Longest 5/2+-power-free word with no factor in  ['0100', '1011']  has length  13
Longest 5/2+-power-free word with no factor in  ['0100', '1101']  has length  28
Longest 5/2+-power-free word with no factor in  ['1011', '1101']  has length  44
 
=============
Computations for Lemma 3
 
Longest 5/2+-power-free word with no factor in  ['0011', '0010']  has length  15
Longest 5/2+-power-free word with no factor in  ['0011', '0100']  has length  31
Longest 5/2+-power-free word with no factor in  ['0011', '0101']  has length  12
Longest 5/2+-power-free word with no factor in  ['0011', '1010']  has length  18
Longest 5/2+-power-free word with no factor in  ['0011', '1011']  has length  15
Longest 5/2+-power-free word with no factor in  ['0011', '1101']  has length  31
Longest 5/2+-power-free word with no factor in  ['0011', '1100']  has length  30
\end{verbatim} 
{\scriptsize
\begin{verbatim}
=============
Computations for Lemma 4
 
Longest 5/2+-power-free word with no factor in  ['0101', '1010', '00100110']  has length  24
Longest 5/2+-power-free word with no factor in  ['0101', '1010', '01001100']  has length  50
Longest 5/2+-power-free word with no factor in  ['0101', '1010', '10011001']  has length  33
Longest 5/2+-power-free word with no factor in  ['0101', '1010', '00110010']  has length  50
Longest 5/2+-power-free word with no factor in  ['0101', '1010', '01100100']  has length  24
Longest 5/2+-power-free word with no factor in  ['0101', '1010', '11001001']  has length  24
Longest 5/2+-power-free word with no factor in  ['0101', '1010', '10010011']  has length  24
Longest 5/2+-power-free word with no factor in  ['0101', '1010', '00110011']  has length  52
Longest 5/2+-power-free word with no factor in  ['0101', '1010', '01100110']  has length  33
Longest 5/2+-power-free word with no factor in  ['0101', '1010', '11001101']  has length  50
Longest 5/2+-power-free word with no factor in  ['0101', '1010', '10011011']  has length  24
Longest 5/2+-power-free word with no factor in  ['0101', '1010', '00110110']  has length  24
Longest 5/2+-power-free word with no factor in  ['0101', '1010', '01101100']  has length  24
Longest 5/2+-power-free word with no factor in  ['0101', '1010', '11011001']  has length  24
Longest 5/2+-power-free word with no factor in  ['0101', '1010', '10110010']  has length  88
Longest 5/2+-power-free word with no factor in  ['0101', '1010', '10110011']  has length  50
Longest 5/2+-power-free word with no factor in  ['0101', '1010', '11001100']  has length  52
 
=============
Computations for Theorem 4
 
Longest 5/2+-power-free word with no factor in  ['1011', '1010']  has length  20
()
\end{verbatim}
}
\normalsize
\nocite{*}
\bibliographystyle{abbrvnat}
\bibliography{Low_complexity_fhpf_words}

\end{document}